\documentclass[12pt]{article}
\usepackage{epsfig}

\usepackage{enumitem}
\setlist{nolistsep}

\usepackage{amssymb}
\usepackage{amsmath}

\usepackage{graphicx}
\usepackage[usenames,dvipsnames]{color}

\newcommand\D{{\cal D}}

\renewcommand\L{{\mathscr L}}

\newtheorem{theorem}{Theorem}[section]
\newtheorem{lemma}[theorem]{Lemma}

\newtheorem{example}[theorem]{Example}

\newtheorem{conjecture}[theorem]{Conjecture}
\newtheorem{result}[theorem]{Result}
\newtheorem{remark}[theorem]{Remark}
\newenvironment{proof}{\noindent{\bf Proof}\hspace{0.5em}}
    { \null  \hfill $\square$ \par}

\usepackage[mathscr]{eucal}

\newcommand{\X}{\mathcal X}
\newcommand{\Y}{\mathcal Y}

\renewcommand{\S}{\mathcal S}
\renewcommand{\P}{\mathcal P}

\newcommand{\ES}{{\mathbb S}}

\newcommand{\ST}{\mathscr S_T}

\renewcommand{\r}{{q}}
\newcommand\rr{{\r^2}}

\newcommand{\li}{\ell_\infty}

\newcommand{\orsps}{order-$\r$-subplanes}
\newcommand{\orsls}{order-$\r$-sublines}
\newcommand{\orsp}{order-$\r$-subplane}
\newcommand{\orsl}{order-$\r$-subline}
\newcommand{\psline}{pencil-subline}
\newcommand{\dcsline}{dual-conic-subline}
\newcommand{\LpiP}{\L_{\pi,P}}

\newcommand{\elltau}{\ell= [-\tau\tau^q, \tau^q+\tau,-1]}

\newcommand{\st}{:}

\newcommand\PGammaL{{\mbox{P}\Gamma {L}}}
\newcommand\PGL{{\rm PGL}}
\newcommand\GF{{\rm GF}}
\newcommand\PG{{\rm PG}}

\newcommand{\car}{E} 

\renewcommand\setminus{\backslash}
\newcommand{\Label}{\label}

\newcommand\todelete[1]{{                 }}

\begin{document}
%
%

\title{Exterior splashes and linear sets of rank 3}

\author{S.G. Barwick and Wen-Ai Jackson
\date{20 April 2014}
}

\maketitle

Corresponding Author: Dr Susan Barwick, University of Adelaide, Adelaide
5005, Australia. Phone: +61 8 8313 3983, Fax: +61 8 8313 3696, email: susan.barwick@adelaide.edu.au






\begin{abstract}
In $\PG(2,q^3)$, let $\pi$ be a subplane of order $q$ that is exterior to $\li$. The exterior splash of $\pi$ is defined to be the set of $q^2+q+1$ points on $\li$ that lie on a line of $\pi$. This article investigates properties of an exterior \orsp\ and its exterior splash. We show that  the following objects are projectively equivalent: exterior splashes, covers of the circle geometry  $CG(3,q)$,  Sherk surfaces of size $q^2+q+1$, and   scattered linear sets of rank 3. We compare our construction of exterior splashes with the projection construction of a linear set. 
We give a geometric construction of the two different families of sublines in an exterior splash, and compare them to the known families of sublines in a scattered linear set of rank 3.
\end{abstract}

\section{Introduction}
Let $\pi$ be a subplane of $\PG(2,q^3)$ of order $q$ that meets $\li$ in 0 or 1 point. Each line of $\pi$, when extended to $\PG(2,q^3)$, meets $\li$ in a point. The set of points of $\li$ that lie on a line of $\pi$ is called the {\em splash} of $\pi$. If $\pi$ is tangent to $\li$ at a point $T$, then the splash of $\pi$ consists of $T=\pi\cap\li$, and $q^2$ further points. We call this splash a \emph{tangent splash}, and these were investigated in \cite{BJ-tgt1,BJ-tgt2}.
If $\pi$ is exterior to  $\li$, then the splash $\ES$ contains $q^2+q+1$ points, and we call it an {\em exterior splash}. 
Note that each line of $\pi$ contains a distinct point of $\ES$, and conversely, each point of $\ES$ lies on a unique line of $\pi$.  
More generally, we can define an exterior splash onto any exterior line. That is, let $\pi$ be a subplane of order $q$ of $\PG(2,q^3)$, and let $\ell$ be an exterior line of $\pi$, then the {\em exterior splash of $\pi$ onto $\ell$} is the set of $q^2+q+1$ points of $\ell$ that lie on lines of $\pi$. 

In this article we will investigate exterior \orsps\ and  exterior splashes.  In Section~\ref{sec:group} we consider the relationship of exterior splashes to other known objects. In particular, exterior splashes are projectively equivalent to  scattered linear sets of rank 3; covers of circle geometries $CG(3,q)$; and Sherk surfaces of size $q^2+q+1$.

The rest of the article focuses on the relationship between an exterior \orsp\ and its exterior splash. 
In Section~\ref{sec:orsl-splash}, we use \orsls\ contained in an exterior splash $\ES$ in terms of an associated \orsp\ $\pi$.
In an analogous manner to \cite[Remark 20]{lavr10}, we also discuss why projective bundles of conics of $\pi$ arise naturally as \orsls\ of $\ES$. 
In Section~\ref{sec:orsl-lin-ext},  we show that the two families of $q^2+q+1$ sublines in an exterior splash are equivalent in some sense, despite their apparent difference in definition.

A $\GF(q)$-linear set of rank 3 can be obtained by projecting an \orsp\ of $\PG(2,q^3)$ onto a line, see \cite{luna04}. 
In Section~\ref{sec:proj-ext}, we investigate the relationship between this projection construction of a scattered linear set of rank 3, and our construction of an exterior splash. In Section~\ref{sec:proj-tgt}, we revisit the tangent splash of a tangent \orsp, and look at the projection in this context.

Finally, in Section~\ref{sec:common-splash}, we look at two exterior \orsps\ that have a common exterior splash, and determine how they can intersect. Further, we show that there are exactly two \orsps\ that have a common exterior splash and share a fixed \orsl.

\section{Notation and definitions}

In this section we introduce the notation we use, as well as defining linear sets, circle geometries and Sherk surfaces. 

If $\pi$ is a subplane of $\PG(2,q^3)$ of order $q$, then we call $\pi$ an {\em\orsp}. If $\li$ is an exterior line of $\pi$, then we say $\pi$ is an {\em exterior} \orsp. An {\em\orsl} of $\PG(2,q^3)$ is a line of an \orsp, that is, it is isomorphic to  $\PG(1,q)$. 
Points in $\PG(2,\r^3)$ have homogeneous coordinates
$(x,y,z)$ with $x,y,z\in\GF(\r^3)$. Let the line at infinity $\li$ have
equation $z=0$; so the affine points of $\PG(2,\r^3)$ have coordinates
$(x,y,1)$. 
We can construct $\GF(q^3)$ as a cubic extension of $\GF(q)$ using a primitive element  $\tau$  with primitive
polynomial 
\begin{eqnarray}x^3-t_2x^2-t_1x-t_0,\label{t0t1t2}
\end{eqnarray} 
where $t_0,t_1,t_2\in\GF(q)$; so every element in $\GF(\r^3)$
can be uniquely written as $a_0+a_1\tau+a_2\tau^2$ with
$a_0,a_1,a_2\in\GF(\r)$. 

We define linear sets of $\PG(1,q^3)$, more generally, linear sets of $\PG(n-1,q^t)$ are defined in \cite{lavr10}. 
The points of $\PG(1,q^3)$ can be considered as elements of   the 2-dimensional vector space $V=\GF(q^3)^2$ over 
$\GF(q^3)$.
 Let $U$ be a subset of $V$ that forms a 3-dimensional vector space over $\GF(q)$. Then the vectors of $U$ (considered as vectors over $\GF(q^3)$) form a $\GF(q)$-linear set of rank 3 of $\PG(1,q^3)$. A \emph{scattered linear set of rank 3} is a $\GF(q)$-linear set of $\PG(1,q^3)$ of rank 3 and size $q^2+q+1$. 
By \cite{lavr10}, all scattered linear sets of rank 3 are projectively equivalent. Scattered linear sets were introduced in \cite{blok00}, and have recently been studied in \cite{dona14,lavr10,lavr13,LMPT14}.

In \cite{bruc73a,bruc73b}, Bruck gave a set of axioms for higher dimensional circle geometries. We look at the 3-dimensional case, namely $CG(3,q)$. The points of $CG(3,q)$ can be identified with the points of $\PG(1,q^3)$, and the \emph{circles} of $CG(3,q)$ are identified with the \orsls\ of $\PG(1,q^3)$. The {\em stability group} of a circle is defined to be the subgroup of Aut$CG(3,q)=\PGammaL(2,q^3)$ that fixes the circle pointwise. For two distinct points of $CG(3,q)$, let $\phi(P,Q)$ denote the group generated by the stability groups of all circles containing $P$ and $Q$. A {\em cover} is defined to be an orbit under $\phi(P,Q)$ of any point $R$ distinct from $P,Q$. In \cite{bruc73b}, it is shown that every cover of $CG(3,q)$ has $q^2+q+1$ points, and can be represented in one of the following two ways, using the identification of the points of $CG(3,q)$ with the field $\GF(q^3)\cup\{\infty\}$: \begin{enumerate}
\item[I.] $\{x\in\GF(q^3)\st N(x-a)=f\}$, for some $a\in\GF(q^3)$, and $f\in\GF(q)\setminus\{0\}$.
\item[II.] $\{x\in\GF(q^3)\cup\{\infty\}\st N\left(\frac{x-a}{x-b}\right)=f\}$, for some $a,b\in\GF(q^3)$, and $f\in\GF(q)\setminus\{0\}$.
\end{enumerate}
where $N$ is the norm from $\GF(q^3)$ to $\GF(q)$, that is $N(x)=x^{q^2+q+1}$. The points $P,Q$ are called the {\em carriers} of the cover. The carriers for a cover of type I are $\{a,\infty\}$, and the carriers for a cover of type II are $\{a,b\}$.
A cover of type II will contain $\infty$ if and only if $f=1$.

Sherk \cite{sher86}, studied objects other than circles and covers in $CG(3,q)$. Representing points of $CG(3,q)$ as elements of $\GF(q^3)\cup\{\infty\}$, a {\em Sherk surface} is the set of points satisfying
$$S(f,\alpha,\delta,g)=\{z\in\GF(q^3)\cup\{\infty\}\st fN(z)+T(\alpha^{q^2}z^{q+1})+T(\delta z)+g=0\}$$ for $f,g\in\GF(q)$, $\alpha,\delta\in\GF(q^3)$, where $N,T$ are the norm and trace from $\GF(q^3)$ to $\GF(q)$.  Sherk surfaces of size $q^2+q+1$ are precisely the Bruck covers of $CG(3,q)$. 

\section{Exterior splashes}\Label{sec:group}

\subsection{An important Singer group}\label{sec:singer-gp}

We will need a particular Singer group acting on  \orsps.  First we define the notion in $\PG(2,q^3)$ of conjugate  points and lines with respect to an \orsp.  Let $\pi$ be an \orsp\ of $\PG(2,q^3)$. There is a unique collineation group of order three that fixes every point of $\pi$, let $\zeta$ be a generator of this group. Then the points of $\PG(2,q^3)\setminus\pi$ can be partitioned into sets of size three of form $\{P, \zeta(P), \zeta^2(P)\}$, called {\em conjugate points with respect to $\pi$}.  Note that $\{P, \zeta(P), \zeta^2(P)\}$ are collinear if and only if they lie on the extension of an \orsl\ of $\pi$.
Similarly,
the lines of $\PG(2,q^3)\setminus\pi$ can be partitioned into sets of size three of form $\ell, \zeta(\ell), \zeta^2(\ell)$, called {\em conjugate lines with respect to $\pi$.}
For example, let $\pi=\PG(2,q)$, then $\zeta\colon (x,y,z)\mapsto (x^q,y^q,z^q)$ fixes every point of $\pi$ and acts semi-regularly on the remaining points of $\PG(2,q^3)$. If $P$ is a point of $\PG(2,q^3)\setminus\PG(2,q)$, then the conjugate points with respect to the \orsp\ $\pi=\PG(2,q)$ are $P,P^q,P^{q^2}$. 
If $\ell$ is a line of $\PG(2,q^3)\setminus\PG(2,q)$, then the conjugate lines with respect to the \orsp\ $\pi=\PG(2,q)$ are $\ell,\ell^q,\ell^{q^2}$. 

We need the following result about  the collineation groups acting on an exterior \orsp. 

\begin{theorem}\Label{propthm}
Consider  the collineation group $G=\PGL(3,q^3)$ acting on $\PG(2,q^3)$.
Let $I=G_{\pi,\ell}$ be the subgroup of $G$
fixing an \orsp\ $\pi$, and a line $\ell$ exterior to $\pi$. 
Then 
\begin{enumerate}[noitemsep,nolistsep]
\item $I$ is cyclic of order $q^2+q+1$, and acts regularly on the points and on the lines of $\pi$.
\item $I$ fixes exactly three lines: $\ell$, and its conjugates $m$, $n$ with respect to $\pi$. 
Further $I$ acts semi-regularly on the remaining line orbits. 
\item $I$ \ fixes exactly three points: $E_1=\ell\cap m$, 
$E_2=\ell\cap n$, 
$E_3=m \cap n$ (which are conjugate with respect to $\pi$) and acts semi-regularly on the remaining point orbits.\\
The points $E_1,E_2$ are called the {\sl carriers} of $\pi$.
\end{enumerate}
\end{theorem}

\begin{proof}
 We first consider the homography $\phi$ of $\PG(2,q^3)$ with matrix  $$
 T=\begin{pmatrix}0&1&0\\0&0&1\\t_0&t_1&t_2\end{pmatrix},$$
%
 where $t_0,t_1,t_2$ are as in (\ref{t0t1t2}) and  $\phi(x,y,z)=T(x,y,z)^t$.  The matrix $T$ has three eigenvalues $\tau$, $\tau^\r$ and $\tau^{\r^2}$ with corresponding eigenvectors $Q=(1,\tau,\tau^2)$, $Q^\r$ and $Q^{\rr}$ respectively.  
 Further, the eigenvalues of $T^i$ are 
 $\tau^{i},\tau^{i\r},\tau^{i\r^2}$ with corresponding  eigenvectors $Q,Q^\r,Q^{\r^2}$ respectively. 
 Now $\tau^i=\tau^{i\r}$ if and only if $\tau^i\in\GF(\r)$, if and only $\r^2+\r+1\bigm| i$, 
hence the eigenvalues are distinct. Thus, for $0<i<q^2+q+1$, $\phi^i$ fixes exactly three points $Q,Q^\r,Q^{\r^2}$. Further, $\phi$ has order $\r^2+\r+1$, and $\phi$ acts semi-regularly on the points other than $Q,Q^\r,Q^{\r^2}$. 
 In particular, $\langle\phi\rangle$ is a Singer cycle of the \orsp\ $\pi=\PG(2,q)$, and so acts regularly on the points and on the lines of $\pi$.
 
Now let $G=\PGL(3,q^3)$, 
it is well known that $|G|=\r^9(\r^9-1)(\r^6-1)$, and that $G$ is transitive on the points of $\PG(2,\r^3)$; on the lines of    $\PG(2,\r^3)$; and on the \orsp s of $\PG(2,\r^3)$.
Hence when considering the subgroup $K=G_\pi$ for some \orsp\ $\pi$, we can without loss of generality let $\pi=\PG(2,q)$.
Let $K=G_\pi=\PGL(2,\r)$. Thus $K$ is transitive on the points and lines of $\pi$, and so this gives one point orbit of $K$.   Further, $|K|=q^3(q^3-1)(q^2-1)$.

We now consider the point $P=(1,1,\tau)$ and calculate its orbit under $K$. Note that $P$ is not in $\pi$, but $P$ is on the  line $[1,-1,0]$, which is a line of $\pi$, so  $P$ lies on exactly one line of $\pi$.   Consider the subgroup of $K$ fixing $P$, so consider a homography with matrix $$C=\begin{pmatrix}
a&b&c\\d&e&f\\g&h&i
\end{pmatrix},
$$ with $a,\ldots,i\in\GF(\r)$, such that $C(1,1,\tau)^t\equiv(1,1,\tau)$.
Then $a+b+c\tau=d+e+f\tau$ (hence $c=f$ and $a+b=d+e$) and 
$(a+b+c\tau)\tau=g+h+i\tau$ giving $c=0$ (so $f=0$),  $a+b=i$, and $g+h=0$.
 Hence $$C=\begin{pmatrix}
a&b&0\\d&a+b-d&0\\g&-g&a+b\end{pmatrix}=\begin{pmatrix}a&1-a&0\\d&1-d&0\\g&-g&1\end{pmatrix}$$  since we can take $a+b=1$ as det $C\neq0$. Now det $C\neq0$ if and only if $a\neq d$. Hence the size of $K_P$ is $\r^2(\r-1)$. Thus by the orbit stabilizer theorem, $|P^K|=|K|/|K_P|=q(q^3-1)(q+1).$ This is the number of points of 
$\PG(2,q^3)\setminus\pi$ that lie on exactly one line of $\pi$, hence these points lie in one orbit of $K$.

We now consider the  points of $\PG(2,q^3)\setminus\pi$ that lie on exactly zero lines of $\pi$.
 Consider the point $Q=(1,\tau,\tau^2)$, 
 it does not lie on any line of $\PG(2,q)$. We calculate the size of the orbit of $Q$ under $K$. Note first that the subgroup of $K$ generated by the homography $\phi$ with matrix $T$  fixes $Q$. 
Now consider a homography of $K$ with matrix $$B=\begin{pmatrix}a&b&c\\d&e&f\\g&h&i\end{pmatrix},$$ $a,\ldots,i\in\GF(q)$ that fixes $Q$. We have $(a+b\tau+c\tau^2)\tau=d+e\tau+f\tau^2$ and $(d+e\tau+f\tau^2)\tau=g+h\tau+i\tau^2$. So we have $B=\big(\boldsymbol{x}, \boldsymbol{x}T, \boldsymbol{x}T^2\big)^t$ where $ \boldsymbol{x}=(a,b,c)$. Now $B$ has nonzero determinant if and only not all of $a,b,c$ are zero. Further, $\boldsymbol{x}=(a,b,c)$ and $\boldsymbol{x}=y(a,b,c)$ correspond to the same collineation if and only if $y\in\GF(\r)$. Thus the subgroup $K_Q$  is of size $\r^2+\r+1$. Hence $K_Q$  is generated by the Singer cycle $\phi$ with matrix $T$. By the orbit stabilizer theorem, $|Q^K|=|K|/|K_Q|=q^3(q^2-1)(q-1)$. This is equal to the number of points of $\PG(2,q^3)$ that lie on zero lines of $\pi=\PG(2,q)$. Hence these points all lie in one orbit of $K$. 
A similar argument shows that $K$ has three line orbits: the lines of $\pi$, the lines meeting $\pi$ in exactly one point, and the lines exterior to $\pi$.

As $K=G_\pi$ is transitive on the exterior lines of $\pi$, so we can without loss of generality consider the line $\ell=[1,\tau,\tau^2]$ which is exterior to $\pi=\PG(2,q)$, and we calculate $I=G_{\pi,\ell}$ for this $\pi,\ell$. 
 Consider the homography $\phi'$ which has matrix $T^{-t}$ (the transpose of the inverse of $T$).  Note that $\phi'$ has order $q^2+q+1$ as  $\phi$ does.
Now $\phi'$ maps a line $[l,m,n]$ of $\PG(2,q^3)$ to the line 
 $[l,m,n]T^t=\ell T^t$.
Hence $\phi'$ fixes the line $\ell=[1,\tau,\tau^2]$.
So using a similar argument to the above paragraph, we have $I=G_{\pi,\ell}$ is generated by the singer cycle $\phi'$. Hence $|I|=q^2+q+1$ and $I$ acts regularly on the points, and on the lines, of $\PG(2,\r)$, proving part 1.

As $T$ has entries in $\GF(\r)$, we have $T=T^\r$, hence 
$\phi'(\ell^q)=\ell^qT^t=(\ell T^t)^q=(\ell)^q=\ell^q$. That is $\phi'$ fixes $\ell^q$, similarly, $\phi'$ fixes $\ell^{q^2}$.
Hence $\phi'$ will fix the three points $P=\ell\cap\ell^{q^2}$, $P^q=\ell\cap\ell^q$, $P^{q^2}=\ell^q\cap\ell^{q^2}$.
A similar argument to that used when analysing the homography $\phi$ shows that $I$ acts semi-regularly on the  remaining line (and point) orbits of $\PG(2,q^3)$, proving parts 2 and 3.
%
%
%
 \end{proof}

The  Singer group $I=\PGL(3,q^3)_{\pi,\ell}$ has a number of important consequences.
We can  use $I$ to define a  
special family of conics and dual conics in  $\pi$ which play an important role in exterior splashes.  A conic of $\pi$ whose extension to $\PG(2,q^3)$ contains the three conjugate points $\car_1,\car_2,\car_3$ is called a $(\pi,\ell)$-{\em special conic} of $\pi$.
Dually we can define a $(\pi,\ell)$-{\em special-dual conic} to be an dual conic of $\pi$ that contains the three conjugate lines $\ell,m,n$. Special conics and dual conics are irreducible. Further, the two incidence structures with points the points of $\pi$, lines the 
$(\pi,\ell)$-special conics (respectively dual conics) of $\pi$, and natural incidence  are isomorphic to $\PG(2,q)$.
In addition, a \emph{projective bundle} of conics is defined in \cite{BBEF} to be a set  of $q^2+q+1$ conics of $\PG(2,q)$ which pairwise meet in exactly one point. 
There are only three known classes of projective  bundles, of which the $(\pi,\ell)$-special conics of $\pi$ form a {\em circumscribed bundle}, and the $(\pi,\ell)$-special-dual conics of $\pi$ form an {\em inscribed bundle}. 

\subsection{An example}

 We will need an  example  of an \orsp\ to prove subsequent results. 
 So the next result takes the \orsp\ $\pi=\PG(2,q)$ of $\PG(2,q^3)$, finds a suitable exterior line $\ell$ to 
$\pi$,  and calculates the exterior splash,  the Singer group $I$, and the carriers for  $\pi$. 

\begin{example}\Label{splash-eg}
Let $\pi=\PG(2,q)$ be an \orsp\ of $\PG(2,q^3)$.
\begin{enumerate}[noitemsep,nolistsep]
\item The  line $\elltau$  is exterior to $\pi$.
\item The exterior splash $\ES$ of $\pi=\PG(2,q)$ on $\elltau$ is
 $$
 \ES=
 \{\car+\theta \car^q\st
\theta^{q^2+q+1}=-1,\theta\in\GF(q^3)\}, \quad {\rm where\ } E=(1,\tau,\tau^2).$$
\item 
The Singer group $I=\PGL(3,q^3)_{\pi,\ell}$ is generated by the  homography $\phi$  with matrix
$$
 T=\begin{pmatrix}0&1&0\\0&0&1\\t_0&t_1&t_2\end{pmatrix},
$$ where $t_0,t_1,t_2$ are as in {\rm (\ref{t0t1t2})}, and $\phi(x,y,z)^t=T(x,y,z)^t$.  The three fixed  points of $\phi$  are   the  points  $E=\ell^{q^2}\cap\ell=(1,\tau,\tau^2)$, $E^q,E^{q^2}$. The three fixed lines of $\phi$ are  $\ell, \ell^q,\ell^{q^2}$. So $\pi$ has carriers
$\car=(1,\tau,\tau^2)$ and $\car^q=(1,\tau^q,\tau^{2q})$.
\end{enumerate}
\end{example}

\begin{proof} We first show that the line $\elltau$ is exterior to $\pi=\PG(2,q)$. 
The points of $\pi$ are of the form $(x,y,z)$,  $x,y,z\in\GF(q)$, not all zero. This point lies on the line 
$\elltau$
 if and only if $\tau\tau^qx-(\tau+\tau^q)y+z=0$ for some $x,y,z\in\GF(q)$, not all zero. As  $1,\tau\tau^q,\tau+\tau^q$ of $\GF(q^3)$ are linearly independent over $\GF(q)$,   $\tau\tau^qx-(\tau+\tau^q)y+z=0$ has no solutions with $x,y,z\in\GF(q)$, not all zero, thus $\pi$ is exterior to $\ell$. 
We can parameterize the line $\ell$ as $\ell=\{E+\theta E^q\st\theta\in\GF(q^3)\cup\{\infty\}\}$, where $E=(1,\tau,\tau^2)$ (so $E$ has parameter $0$, and $E^q$ has parameter $\infty$). 
Consider a line $\ell_i$ of $\pi$, so $\ell_i=[l,m,n]$, $l,m,n\in\GF(q)$, not all zero. Then $\ell_i$ meets the line $\elltau$ in the point $X_i=(m+n(\tau+\tau^q),\ n\tau\tau^q-l,\ -l(\tau+\tau^q)-m\tau\tau^q).$
We can write this as $X_i=\car+\theta_i\car^q$ where 
$\theta_i=-(l+m\tau+n\tau^{2})/(l+m\tau+n\tau^2)^q$.
Note that  
$\theta_i^{q^2+q+1}=\theta_i^{q^2}\theta_i^q\theta_i=-1$. As there are $q^2+q+1$ solutions to the equation $\theta^{q^2+q+1}=-1$, the $q^2+q+1$ lines of $\pi$ meet $\ell$ in the $q^2+q+1$ points $\car+\theta\car^q$ with $\theta^{q^2+q+1}=-1$. That is, 
 $\pi$ has exterior splash $\ES=\{\car+\theta \car^q\st
\theta^{q^2+q+1}=-1,\theta\in\GF(q^3)\}$ on $\ell$, completing the proof of part 2. The proof of part 3 involves straightforward  calculations.
\end{proof}

\subsection{Exterior splashes and linear sets}

 The following result  is important to our study of exterior \orsps\ and exterior splashes.

\begin{theorem}\Label{transsplashes}
Consider the collineation group  $G=\PGL(3,\r^3)$ acting on $\PG(2,q^3)$. The 
subgroup $G_\ell$ fixing a line $\ell$ is transitive on the \orsps\ that are exterior to $\ell$, and is transitive on the
exterior splashes on $\ell$. 
\end{theorem}
\begin{proof}
Recall that the group 
$G=\PGL(3,q^3)$ is transitive on the \orsp s of $\PG(2,\r^3)$.  From the proof of Theorem~\ref{propthm}, the subgroup of $K$ fixing an \orsp\ $\pi$ is transitive on the exterior lines of $\pi$. Hence  $G$ is transitive on pairs $(\pi,\ell)$ where $\pi$ is an \orsp\ and $\ell$ is a line exterior to $\pi$.
Thus $G_\ell$ is transitive on the the \orsps\ that are exterior to $\ell$, and hence $G_\ell$ is transitive on the
exterior splashes of $\ell$.
\end{proof}

The importance of this result is two fold. Firstly, it means that  when proving results about  exterior \orsps\ and  exterior splashes, we can without loss of generality prove results for a particular example \orsp.
Secondly,  this theorem shows that all exterior splashes are projectively equivalent. It is then straightforward to use Example~\ref{splash-eg}, and the result that all scattered linear sets of rank 3 are projectively equivalent by \cite{lavr10}, to prove that  exterior splashes of $\PG(2,q^3)$ are projectively equivalent to scattered linear sets of rank 3. 

\begin{theorem}\Label{lin-set}
Exterior splashes of $\PG(1,q^3)$ are projectively equivalent to  scattered linear sets of rank 3.
\end{theorem}

We note that in concurrent, independent work, Lavrauw and Zanella \cite{lavr14} prove the more general result that  exterior splashes of $\PG(1,q^n)$ are projectively equivalent to scattered linear sets of $\PG(1,q^n)$.

\subsection{Circle geometries and Sherk surfaces}
\Label{sec:def-carrier}

%
%

In this section we observe that exterior splashes are projectively equivalent to covers of the circle geometry $CG(3,q)$, and to Sherk surfaces of size $q^2+q+1$.

\begin{theorem}\Label{spiscov}
Let $\pi$ be an exterior \orsp\ with exterior splash $\ES$ on $\li=\PG(1,q^3)$ and carriers $E_1,E_2$. Then $\ES$ is projectively equivalent to  a cover of the circle geometry $CG(3,q)$  with carriers $E_1,E_2$.
\end{theorem}

\begin{proof} By Theorem~\ref{transsplashes}, we can without loss of generality prove this for the exterior splash $\ES$ of the \orsp\  $\pi=\PG(2,\r)$ onto the exterior line $\elltau$. 
By Example~\ref{splash-eg}, $\ES$ has  exterior splash $\ES=\{\car+\theta \car^q\st
\theta^{q^2+q+1}=-1,\theta\in\GF(q^3)\}$ on $\ell$, and carriers
 $E_1=\car=(1,\tau,\tau^2)$ and $E_2=\car^q$ (corresponding to $\theta=0$ and $\theta=\infty$ respectively).
Hence using the notation of Section 2, the exterior splash is 
a cover of type I with $a=0$, $f=-1$.
\end{proof}

We note that this equivalence can also be deduced from the general theory of linear sets of pseudoregulus type in \cite{LMPT14}. In particular, an exterior splash $\ES$ is equivalent to a scattered linear set of rank 3 of pseudoregulus type. The carriers of $\ES$ are called transversal spaces of the pseudoregulus.  
There are two  important corollaries for exterior splashes, we note that they both follow immediately from properties of circle geometries, and are also proved directly in \cite{LMPT14}. We state them here in the exterior splash context as they are important for understanding the carriers of an exterior splash.

\begin{result}\Label{carriers-belong-splash}
\begin{enumerate}
\item  Let $\pi_1,\pi_2$ be two exterior \orsps\ with the same exterior splash $\ES$ on $\li$, then $\pi_1$ and $\pi_2$ have the same carriers. That is, an exterior splash has two unique carriers. 
\item Given two carriers $\car_1,\car_2$ of $\li$, there are exactly $q-1$ disjoint exterior splashes with carriers $\car_1,\car_2$.
\end{enumerate}
\end{result}

\begin{remark} {\rm As the Sherk surfaces of size $q^2+q+1$ are precisely the Bruck covers of $CG(3,q)$, exterior splashes are also projectively equivalent to Sherk surfaces of size $q^2+q+1$. In fact, the Sherk surface corresponding to the exterior splash of the \orsp\ $\pi=\PG(2,q)$ onto the exterior line $\elltau$ has parameters $S(1,0,0,1)$.
This equivalence means we can adapt results about permutation groups of Sherk surfaces from \cite{sher86} to our setting, also see \cite{dona14}. In particular we will need the following collineation group.  
}\end{remark}

%

\begin{lemma} \Label{spl-gp-eg} Let $\ES$ be the exterior splash of the \orsp\ $\pi=\PG(2,q)$ onto the exterior line $\elltau$.
Then  $\ES\equiv\{(\theta,1)\st\theta^{q^2+q+1}=-1,\theta\in\GF(q^3)\}$, where $(\theta,1)\equiv\theta E_1+E_2$. The group of collineations of $\ell$ fixing $\ES$ has as generators  two homographies $\Gamma$, $\Delta$ with matrices 
\[
\begin{pmatrix}
\tau&0\\0&\tau^q
\end{pmatrix} \quad {\rm and}\quad
\begin{pmatrix}
0&1\\1&0
\end{pmatrix}
\]
respectively; so   $\Gamma$ fixes the carriers $\car_1,\car_2$, and $\Delta$ interchanges them. 
\end{lemma}

\section{Order-$q$-sublines of an exterior splash}\Label{sec:orsl-splash}

We have so far discussed the relation between exterior splashes and other geometrical objects in their own right.  For the remainder of the paper we will primarily consider the relationship between an exterior \orsp\ and its exterior splash.
In this section we investigate the \orsls\ contained in an exterior splash. We
first characterise the  \orsls\ of an exterior splash {\em with respect to} an associated exterior \orsp. Then we compare this with the characterisation of \orsls\ of a scattered linear set discussed in \cite{lavr10}.

\subsection{Sublines in an exterior splash}\Label{sec:orsl-extsplash}

%

%

We first note that the dual of \cite[Corollary 19]{lavr10} can be generalised to the following result.

\begin{lemma}\Label{secondlines}
Let $\pi$ be an exterior \orsp\ of $\PG(2,q^3)$ with exterior splash $\ES$ on $\li$, and let
 $\Gamma$ be a dual $(q+1)$-arc of $\pi$. The lines of $\Gamma$ meet the exterior splash $\ES$ in an \orsl\ if and only if $\Gamma$ is $(\pi,\li)$-special-dual conic of $\pi$.
\end{lemma}

Using this lemma, it is straightforward to prove the following
geometric interpretation of the \orsls\ of $\ES$ in relation to an associated \orsp. %

\begin{theorem}\Label{defn:sline}
 Let $\pi$ be an exterior \orsp\ of $\PG(2,q^3)$, $q>2$, with exterior splash $\ES$ on $\li$.  The $2(q^2+q+1)$ \orsls\ of $\ES$ lie  in two families of size $q^2+q+1$ as follows.
\begin{enumerate}[noitemsep,nolistsep]
\item If $A$ is a point of $\pi$, then the pencil of $q+1$ lines of $\pi$ through $A$ meets $\li$ in an \orsl\ of $\ES$, called a {\em $\pi$-\psline.} 
\item If $\Gamma$ is a $(\pi,\li)$-special-dual conic of $\pi$, then the lines of $\Gamma$ meet $\li$ in an \orsl\ of $\ES$, called a {\em $\pi$-\dcsline}. 
\end{enumerate}
\end{theorem}



\begin{remark}{\rm
The $\pi$-\dcsline s of $\ES$ arise from an inscribed bundle of conics (see Section~\ref{sec:singer-gp}). We note that in  \cite[Remark 20]{lavr10}, the authors show that irregular sublines of a scattered linear set of rank 3 arise from a circumscribed bundle of conics.
}\end{remark}

\subsection{Sublines of a linear set and an exterior splash}\Label{sec:orsl-lin-ext}

In this section we compare the construction of the two families of \orsls\  from \cite{lavr10} with that of Theorem~\ref{defn:sline}. In particular, we explain why the notions of regular and irregular \orsls\ of a linear set defined in \cite{lavr10} are not intrinsic properties of the linear set, but that the two classes can be interchanged.



We first prove that the characterisation of \orsls\ given in Theorem~\ref{defn:sline}  is a property of the associated \orsp, not a property of the exterior splash. 

\begin{theorem}\Label{orsls-same}
Consider an exterior splash $\ES$ of an exterior \orsp\ $\pi$ onto $\li$.  Let $\X$ denote the $\pi$-\psline s of $\ES$, and let $\Y$ denote the $\pi$-\dcsline s of $\ES$.  Then there exists an exterior  \orsp\ $\pi'$ with the same exterior splash $\ES$ on $\li$, such that  $\X$ contains the  $\pi'$-\dcsline s of $\ES$ and $\Y$ contains the  $\pi'$-\psline s of $\ES$.
\end{theorem}

\begin{proof} By Theorem~\ref{transsplashes}, all exterior splashes are projectively equivalent, so we can without loss of generality look at the exterior splash $\ES$ of  Example~\ref{splash-eg}.
Now $\ES$ contains the points $X_i=E+\theta_iE^q$, so points in $\ES$ have coordinates equivalent to $(\theta_i,1)\equiv(
l+m\tau+n\tau^{2},-(l+m\tau+n\tau^2)^q)$ where $l,m,n\in\GF(q)$, not all zero.  We can map this to the point $(l+m\tau+n\tau^{2},(l+m\tau+n\tau^2)^q,0)$, so, we can without loss of generality let $\ES=\{(x,x^q,0)\st x \in \GF(q^3)\}$.  (We note that this is equivalent to the canonical form  for a linear set of pseudoregulus type given in \cite{LMPT14}.)
 It is straightforward to calculate the \orsls\ of $\ES$, they are
$\X=\{x_P\st P\in\PG(2,q)\}$,
$\Y=\{y_P\st P\in\PG(2,q)\}$,
where 
\begin{eqnarray*}
 x_P&=&\{ \left(\,l+m\tau+n\tau^{2},(l+m\tau+n\tau^2)^q,0\right)\st  l,m,n\in\GF(q),\ P\cdot[l,m,n]=0\}\\
y_P&=&\{\left((l+m\tau+n\tau^2)^q, l+m\tau+n\tau^{2},0\right)\st  l,m,n\in\GF(q),\ P\cdot[l,m,n]=0\}.
\end{eqnarray*}
 Let $\pi$ be an exterior \orsp\ with exterior splash $\ES$. 
 Consider the involutary homography $\Delta'$ of $\PG(2,q^3)$ with matrix
\[
\begin{pmatrix}
0&1&0\\
1&0&0\\
0&0&1
\end{pmatrix},
\]
 it fixes $\li$. As the homography $\Delta$ of $\li$ given in Lemma~\ref{spl-gp-eg}  exchanges the two families $\X$, $\Y$, of \orsl s of $\ES$, $\Delta'$ exchanges the two sets of \orsl s in $\ES$, that is, $\Delta'$ maps $\X$ to $\Y$ and $\Y$ to $\X$. Thus $\Delta'$ does
  not fix $\pi$, as  only elements of the Singer group $I=\PGL(3,q^3)_{\pi,\li}$ fix $\pi$ and $\li$, and $I$ acts regularly on the elements of $\X$ by Theorem~\ref{propthm}.  Thus $\Delta'$ fixes $\ES$ and maps $\pi$ to another exterior \orsp\ $\pi'$. Hence $\Delta'$ maps the 
 $\pi$-\psline s to the  $\pi'$-\psline s.  Thus if $\X$ contains the $\pi$-\psline s, then $\Y$ contains the $\pi'$-\psline s, as required.
 \end{proof}

Note that there is an analogous argument for scattered linear sets.  Suppose that the exterior splash/scattered linear set $\ES$ in the above proof is the projection of an \orsp\ $\alpha$ from a point $P$ onto $\li$, then under the homography $\Delta'$, we obtain the same linear set $\ES$ as the projection of an \orsp\ $\alpha'$ from a point $P'$, and the \orsl s which are regular with respect to  $\alpha$ are now irregular with respect to  $\alpha'$. That is, the notion of regular and irregular sublines in a fixed scattered linear set $\ES$ of rank 3
can be interchanged by considering a different 
associated \orsp\ that projects the linear set $\ES$.

\section{Projection of \orsp s}\Label{sec:project}

In this section we look further at the relationship between exterior splashes and linear sets. Lunardon and Polverino \cite{luna04} showed that  a $\GF(q)$-linear set of rank 3 of $\PG(1,q^3)$ can be constructed by projecting an \orsp\ $\alpha$ of $\PG(2,q^3)$ from a point $P\notin\alpha$ onto $\li\cong\PG(1,q^3)$. 
We compare this with our definition of the splash of an \orsp\ $\pi$ onto $\li$, namely the intersection of the lines of $\pi$ with $\li$. 
If $\pi$ is exterior to $\li$, then the exterior splash of $\pi$ is equivalent to a  $\GF(q)$-linear set of rank 3 and size $q^2+q+1$. If $\pi$ is tangent to $\li$, then \cite[Theorem 7.1]{BJ-tgt1} shows that the tangent splash of $\pi$ is equivalent to a  $\GF(q)$-linear set of rank 3 and size $q^2+1$. 
This leads to the question: given an \orsp\ $\pi$, can the splash of $\pi$ on $\li$ be the same set of points as the projection of $\pi$ onto $\li$ from some point? 
We consider the case when $\pi$ is exterior to $\li$ in Section~\ref{sec:proj-ext}, and 
 consider the case when $\pi$ is tangent to $\li$ in Section~\ref{sec:proj-tgt}.

\subsection{Projection and exterior splashes}
\Label{sec:proj-ext}

In this section we  explain under what circumstances the exterior splash of an exterior \orsp\ $\pi$ is equal to the projection of $\pi$ onto $\li$. We also discuss  a conjecture regarding which exterior splashes of $\li$ can be obtained by projecting  a fixed exterior plane \orsp\ onto $\li$. 

We begin by showing that the Singer cycle of Theorem~\ref{propthm} acts regularly on the special conics, and special-dual conics of an \orsp.

\begin{theorem}\Label{cgreg}
Let $\pi$ be an \orsp\ of $\PG(2,q^3)$ exterior to $\li$.
The group $I=\PGL(3,q^3)_{\pi,\li}$ acts regularly on the set of $(\pi,\li)$-special conics of $\pi$, and acts regularly on the $(\pi,\li)$-special-dual conics of $\pi$.  
\end{theorem}

\begin{proof}
By Theorem~\ref{propthm}, $I$ fixes $\pi$ and hence acts on the projective plane $\P$ with points  the points of $\pi$, and lines the set $\mathscr C$ of $(\pi,\li)$-special conics of $\pi$.  Consider the orbit $\mathscr C^I$ of $\mathscr C$ under $I$.  As $I$ is transitive on the points of $\pi$ there is a constant number $n$ of elements of $\mathscr C^I$ through each point of $\pi$.   Count the pairs $(Q,\D)$ where $Q$ is a point of $\pi$ on a $(\pi,\li)$-special conic $\D\in \mathscr C^I$.  We have
\[
(\r^2+\r+1)\times n=|\mathscr C^I|\times (\r+1).
\]
As $\r^2+\r+1$ and $\r+1$ have no common factors, it follows that $\r+1$ divides $n$. However, as $\P$ is a projective plane, $n\le \r+1$. Hence  $n=\r+1$ and so $|\mathscr C^I|=\r^2+\r+1$.
Similarly, $I$ acts regularly on the $(\pi,\li)$-special-dual conics of $\pi$.  
\end{proof}

\begin{theorem} \Label{proj-spl}
Let $\pi$ be an exterior \orsp\ with exterior splash $\ES$ on $\li$,  carriers $E_1,E_2$, and third conjugate point $E_3$. Let $P$ be a point of $\PG(2,q^3)\setminus\pi$, then the projection of $\pi$ from $P$ onto $\li$ is equal to $\ES$ if and only if $q$ is even, and $P=E_3$. 
\end{theorem}

\begin{proof} By Theorem~\ref{transsplashes}, we can without loss of generality prove this for the \orsp\ $\pi=\PG(2,q)$
with exterior splash $\ES$ onto the exterior line $\elltau$. 
 By Theorem~\ref{propthm}, the group $I=\PGL(3,q^3)_{\pi,\ell}$  fixing $\pi$ and $\ell$ is cyclic of order $q^2+q+1$; acts regularly on the points and lines of $\pi$; and by Theorem~\ref{cgreg} acts regularly on the $(\pi,\ell$)-special-dual conics of $\pi$.
By Theorem~\ref{defn:sline}, there are $2(q^2+q+1)$ \orsl s in $\ES$, divided into two families denoted ${\mathcal X}$ and ${\mathcal Y}$, one  family contains the 
$\pi$-\psline s of $\ES$, and  the other family contains the $\pi$-\dcsline s. So $I$ acts regularly on the \orsls\ of $\X$, and acts regularly on the \orsls\ of $\Y$. 

Let $P$ be a point of $\PG(2,q^3)$ not in $\pi$ or $\ell$, and let $\L_P$ be the projection of $\pi$ onto $\ell$ from $P$. Suppose that $\L_P=\ES$, that is, the projection of $\pi$ from $P$ is the same as the exterior splash of $\pi$ onto $\ell$. 
By \cite{lavr10}, the
$q^2+q+1$ \orsl s of $\pi$ are projected onto \orsl s of $\mathcal L_P=\ES$ that lie in the same family, ${\mathcal X}$ say.  Further, the line joining $P$ to an \orsl\ $d$ in the other family ${\mathcal Y}$ meets $\pi$ in a set of points which form a conic of $\pi$ whose extension to $\PG(2,q^3)$ contains $P$, and hence $P^q$ and $P^{q^2}$.  That is, it is a $(\pi,PP^q)$-special conic, and so the $q^2+q+1$ $(\pi,PP^q)$-special conics of  $\pi$ are projected to the $q^2+q+1$ \orsl s in ${\mathcal Y}$.

As the group $I$ acts regularly on the \orsls\ of $\Y$,  it induces a Singer cycle   on the $q^2+q+1$ $(\pi,PP^q)$-special conics of  $\pi$.  All these conics have $P,P^q,P^{q^2}$ in common.  If $P$ is not a fixed point of this Singer cycle, then the $q^2+q+1$ $(\pi,PP^q)$-special conics all have the images of $P$ in common, a contradiction as  two distinct irreducible conics have at most four points in common. Thus $P$ is a fixed point of the Singer cycle, and $P\notin\ell$, so  $P=\ell^q\cap\ell^{q^2}$. That is, if $\L_P=\ES$ then $P=\ell^q\cap\ell^{q^2}=(1,\tau^{q^2},\tau^{2q^2})$.

It remains to show that for the point $P=\ell^q\cap\ell^{q^2}$, $\L_P=\ES$ if and only if $q$ is even. 
By Example~\ref{splash-eg}, we can parameterise the points of $\ES$ as $E+\theta E^q$ for $\theta\in\GF(q^3)$, $\theta^{q^2+q+1}=-1$. 
Let $(x,y,z)$, $x,y,z\in\GF(q)$, be a point of $\pi$. The line joining this point to 
 $P=(r,s,t)$ has coordinates
$[sz-ty,\ tx-rz,\ ry-sx]$. It meets the line  $\elltau$ in the point $E+\theta E^q$ with parameter 
%
%
\begin{eqnarray*}
\theta&=&
\frac{x(s\tau^2-t\tau)+y(t-r\tau^2)+z(r\tau-s)}
{x(t\tau^q-s\tau^{2q})+y(r\tau^{2q}-t)+z(s-r\tau^q)}\ =\ \frac{f(x,y,z)}{g(x,y,z)}.
\end{eqnarray*}
If $P=(r,s,t)=(1,\tau^{q^2},\tau^{2q^2})$ we see that $f^{q^2}=g$, and so $\theta^{q^2+q+1}=1$.  Now $\ES$ is the set of points with parameter $\theta$, where $\theta^{q^2+q+1}=-1$, hence the projection $\L_P$  from $P=(1,\tau^{q^2},\tau^{2q^2})$ is  the same as the  exterior splash $\ES$ of $\pi$  onto $\ell$ if and only if $q$ is even. \end{proof}

Suppose we fix an \orsp\ $\pi$ and an exterior line $\ell$. We consider the set $\mathscr S$ of all the projections of $\pi$ from a point $P\in\PG(2,q^3)\setminus\{\pi,\ell\}$ onto $\ell$. Note that the projections in $\mathscr S$ are either tangent or exterior splashes.   As there are less  than $q^6$ possible points to project $\pi$ from, and more than $q^6$ exterior splashes, we will not obtain all the possible exterior splashes on $\ell$. We further examine this
situation. 
As usual, let $I=\PGL(3,q^3)_{\pi,\ell}$ with fixed points $E_1,E_2,E_3$ and fixed lines $\ell,m,n$. Let $J$ be the subgroup of $\PGL(3,q^3)$ that fixes $E_1,E_2,E_3$. Then it is straightforward to show that $I$ is a normal subgroup of $J$, and that further $J$ has seven orbits, namely $E_1$; $E_2$; $E_3$; $\ell\setminus\{E_1,E_2\}$; $m\setminus\{E_1,E_3\}$; $n\setminus\{E_2,E_3\}$; and the points $\PG(2,q^3)\setminus\{\ell,m,n\}$. Moreover, the orbits of $I$ on $\ell$ are $E_1,E_2$ and $q-1$ exterior splashes, with similar orbits on $m$ and $n$. Finally the orbits of $I$ on $\PG(2,q^3)\setminus\{\ell,m,n\}$ consist of \orsps. We now look at the action of $I$ on the splashes in $\mathscr S$.
 
 \begin{lemma}
 \begin{enumerate}
\item  An exterior splash in $\mathscr S$ is either fixed under $I$, or has an orbit of size $q^2+q+1$ under $I$.
\item A  tangent splash  in $\mathscr S$ has an orbit of size $q^2+q+1$ under $I$.
\end{enumerate}
\end{lemma}
\begin{proof} For part 1, consider an exterior splash $\ES\in\mathscr S$ with an orbit under $I$ of size less than $|I|=q^2+q+1$, so $I_\ES$ is non-trivial. We show that in this case $\ES$ is fixed by $I$. Let $g\in I_\ES$, $g$ not the identity, so $g$ fixes $E_1$ and $E_2$.  As $I$ is cyclic, it is abelian, so $\langle g\rangle$ is a normal subgroup of $I$.  Thus  $\langle g\rangle$ fixes either all the points in an orbit of $I$, or has no fixed points on an orbit of $I$.  In the first case, $g$ fixes at least $(q^2+q+1)+2$ points of $\ell$, which is more points that an \orsl\ of $\ell$ contains, hence $g$ fixes $\ell$ pointwise, and so $g$ is the identity on $\ell$.  As $I$ is faithful on $\ell$, it follows that $g$ is the identity.  Thus the second case occurs.
From Lemma~\ref{spl-gp-eg} we can show that the full group (acting faithfully) on the splash is of size $2(q^2+q+1)$, consisting of $q^2+q+1$ involutions and the cyclic group $I$ of odd order, whose elements only fix the carriers of the splash.  As $g\in I_{\ES}\subset I$, it follows that $g$ has odd order and the only fixed points are the carriers of $\ES$.  This implies that $\ES$ has carriers $E_1,E_2$, and is fixed by $I$.  So an exterior splash in $\mathscr S$  has orbit of size 1 or  $|I|=q^2+q+1$.

For part 2, as the exterior splash of $\pi$ on $\ell$ does not contain $E_1$ or $E_2$, by definition no line of $\pi$ contains $E_1$ or $E_2$, thus the tangent splash obtained by projecting $\pi$ from a point on a line of $\pi$ does not have centre $E_1$ or $E_2$.  Thus the centre of the tangent splash is not a fixed point of $I$,  hence its orbit under $I$ is of size $|I|$, and hence the $q^2+q+1$ tangent splash projections  in $\mathscr S$  have different centres, and so are distinct.  
\end{proof}

\begin{lemma}
Let $\ES$ be an exterior splash in $\mathscr S$  with carriers $E_1,E_2$.  Then if $q$ is even, $\ES$ is the projection of $\pi$ from 1 point or $q^2+q+1$ points; and if $q$ is odd, 
$\ES$ is the projection of $\pi$ from $q^2+q+1$ or $q^2+q+2$ points.
\end{lemma}
\begin{proof}  We can without loss of generality consider the \orsp\ $\pi=\PG(2,q)$ and the exterior line $\elltau$.  First note that the projection of $\pi$ from $E_3$ is an exterior  splash with carriers $E_1,E_2$. 
 As $\ES$ occurs as a projection of $\pi$, either it is from $E_3$ (this would give the 1 point or the $q^2+q+2$ points case), or it is from a point $P$ belonging to an orbit $\theta$ under $I$ of size $q^2+q+1$.  In this case, as $\ES$ and $\pi$ are both fixed by $I$, it follows that every point of the orbit $\theta$ projects $\pi$ on to the same splash $\ES$.  Now suppose there is another point $W$ which also projects $\pi$ onto $\ES$. Recall that the exterior splash $\ES$ has two families of \orsl s, $\X$ and $\Y$. An \orsl\  $m$ of $\pi$ is projected by $P$ onto an \orsl\ in one of these families,  $\X$ say.  Further, the images of $m$ under $I$ are projected onto \orsls\ in $\X$. Thus $W$ cannot project $m$ onto an \orsl\ of $\X$, since the point of projection of an \orsl\ is unique.  Thus $W$ projects $m$ onto an \orsl\ $y$ of the family $\Y$.  Similar to the proof of Theorem~\ref{proj-spl}, the lines joining $W$ to the points of $y$  meet $\pi$ in a set of points which form a conic of $\pi$ whose extension to $\PG(2,q^3)$ contains $W$, and hence $W^q$ and $W^{q^2}$. Then, similar to the proof of Theorem~\ref{proj-spl}, it follows that $W$ is a fixed point of the Singer cycle $I$, and so $W=\ell^q\cap\ell^{q^2}$. By Theorem~\ref{proj-spl}, if $q$ is odd, then there is at most one exterior splash with carriers $E_1,E_2$ with $q^2+q+2$ points of projection, and if $q$ is even, all the exterior splashes with carriers $E_1,E_2$  have at most $q^2+q+1$ points of projection.
\end{proof}

We conject that the following fully describes the projection of an \orsp\  onto an exterior line.

\begin{conjecture} Let $\pi$ be an \orsp\ in $\PG(2,q^3)$ with exterior line  $\ell$. Let $\mathscr S$ be the set of all the projections of $\pi$ from a point $P\in\PG(2,q^3)\setminus\{\pi,\ell\}$ onto $\ell$.  Note that the projections in $\mathscr S$ are either tangent or exterior splashes. More specifically,
\begin{enumerate}
\item The points $P\in\PG(2,q^3)\setminus\{\pi,\ell\}$ that lie on the extension of a line of $\pi$ project $\pi$ onto distinct tangent splashes. Hence there are $(q^2+q+1)(q^3-q-1)$ such tangent splashes in $\mathscr S$.
\item The exterior splashes in $\mathscr S$ can be divided into three distinct groups. Let $\ES_0$ be the exterior splash of $\pi$ onto $\ell$, and let $\ES_1,\ldots,\ES_{q-2}$ be the distinct exterior splashes on $\ell$ that have the same carriers $\car_1,\car_2$ as $\ES_0$. 
\begin{enumerate}
\item 
\begin{enumerate}
\item If $q$ is even, then $\ES_0$ is a projection of $\pi$ from $E_3$, and $\ES_i$, $i=1,\ldots,q-2$ is a projection from exactly $q^2+q+1$ points which lie in one orbit of $I$.
\item If $q$ is odd, then $\ES_i$, $i=1,\ldots,q-2$ is the projection of $\pi$ from exactly $q^2+q+1$ points  which lie in one orbit of $I$, with the exception of one of these which is also a projection of $\pi$ from $E_3$.  

\end{enumerate}
\item The remaining exterior splashes in $\S$ are either:
\begin{enumerate}
\item projections of $\pi$ from exactly two  points which are not in the same orbit under $I$; or 
\item projections of $\pi$ from exactly one  point.
\end{enumerate}\end{enumerate}
\end{enumerate}
\end{conjecture}

\subsection{Projection and tangent splashes}\Label{sec:proj-tgt}

Let $\pi$ be an \orsp\ that is tangent to $\li$ at the point $T$, that is, $\pi$ is a tangent \orsp. The lines of $\pi$ meet 
$\li$ in a set $\ST$ of $q^2+1$ points (including $T$) called a {\em tangent splash} with {\em centre} $T$. 
For a detailed study of tangent splashes, see \cite{BJ-tgt1}, in particular, by
\cite[Theorem 7.1]{BJ-tgt1}, tangent splashes are projectively equivalent to $\GF(q)$-linear sets of rank 3 and size $q^2+1$.
We can construct such a linear set by projecting an \orsp. We show that the linear set obtained by projecting  a tangent \orsp\ $\pi$ onto $\li$ can never equal the tangent splash of $\pi$ onto $\li$.


\begin{lemma}\Label{proj-count} Let $\ell,m$ be lines of $\PG(2,q^3)$, and $b$ an \orsl\ of $m$ disjoint from $\ell$. Then each \orsl\ of  $\ell$, disjoint from $m$, is the projeciton of $b$ from exactly one point $P$ not on $\ell$ or $m$. 
\end{lemma}

\begin{proof}
We first show that  the subgroup $K$ of $\PGL(2,q^3)$ (acting on $\PG(1,q^3)$) fixing a point $P$ of $\PG(1,q^3)$ is regular on the \orsl s not containing $P$.
Note that  $\PGL(2,q^3)$ acts faithfully on an \orsl\ $b$ of $\PG(1,q^3)$, as any such homography fixing $3$ points is the identity. Thus $|\PGL(2,q^3)_b|=|PGL(1,q)|=q(q^2-1)$.
As $\PGL(2,q^3)$ is sharply 3-transitive on the points of $\PG(1,q^3)$, it is transitive on the \orsl s of $\PG(1,q^3)$.  Consider the subline $b=\PG(1,q)$ of $\PG(1,q^3)$. From above, the subgroup $H$ of $\PGL(2,q^3)$ fixing $b$ is of order $q(q^2-1)$.  If we consider a homography acting on $P=(\tau,1)^t$, a straightforward calculation shows that $H_P$ is just the identity, hence by the orbit stabilizer theorem, $H$ is transitive on the points of $\PG(1,q^3)$ not on $b$.
Now suppose $P$ is a point not on $b$. If there is an element of $K=\PGL(2,q^3)_P$ which fixes $b$, then this element belongs to $H$ and fixes the point $P$ outside $b$, and so is the identity.  Thus $K$ acts semiregularly on the $q^3(q^3-1)$  \orsl s not through $P$, but as $|K|=q^3(q^3-1)$, this action is transitive, and hence regular.

Consider the group $L$ of axial homographies with axis $m$ and centre on $\ell$. There are $q^3-1$ non-identity elations in $L$, and $q^3(q^3-2)$  non-identity homologies in $L$, so $|L|=q^3(q^3-1)$.  As an element of $L$ fixes $\ell$ and $P=\ell\cap m$, it induces a homography $\sigma$ on $\ell$ which fixes $P$. As the subgroup $K$ above is regular on the \orsl s not containing $P$, $\sigma$ is either the identity, or acts semi-regularly on the \orsl s of $\ell$ not through $P$.  If $\sigma$ acts as the identity on $\ell$, then as it has axis $m$, it follows that it is the identity on $\PG(2,q^3)$.  Hence $L$ acts semi-regularly and hence regularly on the \orsl s of $\ell$ not through $P$.

Now consider a point $X$ projecting $b$ onto an \orsl\ $c$ on $\ell$.  Under the action of $L$, $b$ is fixed pointwise (as $b$ is contained in the axis of the elements of $L$) and $c$ is mapped to the $q^3(q^3-1)$ \orsl s on $\ell$ not through $P$.  Further, as $L$ acts semi-regularly on the $q^3(q^3-1)$ points of $\PG(2,q^3)$ not on $\ell\cup m$, it follows that $L$ is regular on the points not on $\ell\cup m$.  Thus distinct points of $\PG(2,q^3)$ not on $\ell\cup m$ project $b$ onto distinct \orsl s of $\ell$.
\end{proof}

\begin{theorem}\Label{proj-tgt}
Let $\pi$ be a  tangent \orsp\ with tangent splash $\ST$ on $\li$.
Then $\ST$ is not the projection of $\pi$ onto $\li$ from a point $P$. 
\end{theorem}

\begin{proof} We let $\L_{X,P}$ denote the projection of the set of points $X$ of $\PG(2,q^3)$ onto $\li$ from the point $P$. 
Note that if we project the tangent \orsp\ $\pi$ onto $\li$ from a point $P\in\li$ or $P\in\pi$, then $\LpiP$ does not have enough points to be the tangent splash $\ST$ of $\pi$. Similarly, if $P$ is a point not on a line of $\pi$, then $\LpiP$ has too many points to be a tangent splash.
Suppose we project from a point $P\notin\pi$ that is on a line $\ell$ of $\pi$ not through the centre $T$ of  $\ST$, and suppose $\LpiP=\ST$. Let $b$ be an \orsl\ of $\pi$ not through $T$, and not on $\ell$, then $\L_{b,P}$ is an \orsl\ of $\ST$ not through the centre $T$, contradicting \cite[Corollary 7.3]{BJ-tgt1}. So this case cannot occur.

So, suppose we project $\pi$ from a point $P\notin\pi$ that lies on a line of $\pi$ through the centre $T$, and suppose $\LpiP=\ST$. 
Let $J$ be the subgroup of $\PGL(3,q^3)_\pi\cong\PGL(3,q)$ that fixes $T$ and $\li$.
Consider $J_P$, the subgroup of $J$ fixing the projection point $P$.
From the proof of \cite[Theorem 4.2]{BJ-tgt1}, 
$J$ contains only central collineations with centre $T$ and axes belonging to $\pi$. So $P$ is fixed by a non-identity element $\sigma$ of $J$ only if it lies on the axis of $\sigma$.  So $J_P$ is the set of $q$ elations with centre $T$ and axis $TP$.  Hence $|P^J|=|J|/|J_P|=q(q-1)$.

Let  $b$ be an \orsl\ of  $\pi$ not through $T$. Consider the projection $c=\L_{b,P}$.  If $Q$ is another point in $P^J$, then  by Lemma~\ref{proj-count}, it follows that $\L_{b,Q}$  is not $c$.  Thus the projections of $b$ under the $q(q-1)$ different points in $P^J$ is the orbit of $c$ under $J$, and is of size $q(q-1)$.

However, each point $A$ of $\pi$ determines a unique \orsl\ of $\ST$ via the intersection of the lines of $\pi$ through $A$ with $\li$, and all the \orsl s in the tangent splash $\ST$ are determined this way (by \cite[Corollary 7.4]{BJ-tgt1}). As the points of $\pi$ under $J$ fall into $q+1$ orbits, each consisting of $q$  points lying on a line through $T$, it follows that $|c^J|=q$, contradicting the previous paragraph.
\end{proof}

\section{Order-$q$-subplanes with a common exterior splash}\Label{sec:common-splash}

In this section we  investigate the intersection of two exterior \orsps\ that have a common exterior splash. 
We begin with a counting result.
\begin{theorem}\Label{propsplash}
Given an exterior splash $\ES$ on $\li$, there are $2\r^6(\r^3-1)$ \orsps\ with exterior splash $\ES$.
\end{theorem}

\begin{proof} By Theorem~\ref{spiscov},
 the number of exterior splashes is equal to the number of covers of a circle geometry $CG(3,q)$, which is $\frac12\r^3(\r^3+1)(\r-1)$.
Counting quadrangles gives that there are $q^9(q^3-1)(q^3+1)(q-1)$ \orsps\ exterior to a given line $\li$. Further,  by Theorem~\ref{transsplashes}, the group of homographies of $\PG(2,q^3)$ fixing $\li$  is transitive on the \orsp s exterior to $\li$. Hence the number of \orsps\ with a fixed exterior splash is the total number of \orsp s exterior to $\ell$ divided by the number of exterior splashes on $\ell$, that is, $2q^6(q^3-1)$.
\end{proof}

\begin{theorem}\Label{ext2}
Let $\ES$ be an exterior splash of $\li$ and let $\ell$ be an \orsl\ exterior to $\li$, whose extension to $\PG(2,q^3)$ contains a point of $\ES$.  Then there are exactly two \orsps\ that contain $\ell$ and have exterior splash $\ES$.
\end{theorem}

\begin{proof}
Let $\pi$ be an \orsp\ with exterior splash $\ES$, and let $\ell$ be a line of $\pi$. Consider the subgroup $I=G_{\pi,\li}$ of $G=\PGL(3,q^3)$ fixing $\pi$ and $\li$. By 
 Theorem~\ref{propthm} the group $I$ fixes $\ES$ and is transitive on the points of $\ES$.  By \cite[Theorem 2.5]{BJ-tgt2}, the subgroup of $G$ fixing $\ell$ point-wise is transitive on the exterior \orsl s on a line $m\ne\ell$. Thus, given an exterior splash $\ES$, the group fixing $\ES$ is transitive on the exterior \orsl s lying on a line through a point of $\ES$.  Hence the number $n$ of exterior \orsp s containing an \orsl\ $\ell$ lying on a line through a point of a given exterior splash $\ES$ is a constant.  Thus if we count the pairs $(\pi,\ell)$ where $\pi$ is an exterior \orsp\ with a given exterior splash $\ES$, containing any of the $\r^3(\r^3-1)$ \orsl s $\ell$ lying on any of the $\r^3$ lines through a point of a given exterior splash $\ES$, using Theorem~\ref{propsplash}, we get
\[ 2\r^6(\r^3-1)\times(\r^2+\r+1)=(\r^2+\r+1)\times\r^3\times \r^3(\r^3-1)\times n 
\]
and so $n=2$ as required.
\end{proof}

Recall that three collinear points lie in a unique \orsl, and a quadrangle lies in a unique \orsp. Hence 
two \orsps\ in $\PG(2,q^3)$ can meet in either $0,1,2$ or $3$ points, or in $q+1$ points lying on an \orsl,  or in $q+2$ points containing an \orsl.  
In Lemma~\ref{ext2}, we showed that there are exactly two exterior \orsps\ that share an \orsl\ and have the same exterior splash.
We now show that these two \orsps\ do not have any further common points.

\begin{theorem}\Label{orsp-common-splash}
Let $\pi_1$, $\pi_2$ be two exterior \orsps\ with a common exterior splash $\ES$ on $\li$. Then $\pi_1,\pi_2$ meet in at most $q+1$ points, that is, $\pi_1,\pi_2$ meet in $0,1,2$ or $3$ points, or in an \orsl.
\end{theorem}

\begin{proof} By Theorem~\ref{transsplashes} and Example~\ref{splash-eg}, we can without loss of generality let $\pi_1=\PG(2,q)$ with exterior splash $\ES_1=\{\car_1+\theta \car_2\st
\theta^{q^2+q+1}=-1,\theta\in\GF(q^3)\}$ on the exterior line $\elltau$ where $\car_1=(1,\tau,\tau^2)$, $\car_2=\car_1^q=(1,\tau^q,\tau^{2q})$ are the carriers of $\ES_1$. 
Let $\pi_2$ be an \orsp\ exterior to $\ell$ with exterior splash $\ES_2$ on $\ell$. 
We show that if $\ES_1=\ES_2$ and  $\pi_1,\pi_2$ contain $q+2$ common points, then $\pi_1=\pi_2$.
By Theorem~\ref{propthm}, as the group $I$ fixing $\pi_1$ and $\ell$ (and hence $\ES_1$) is transitive on the lines of $\pi_1$, without loss of generality we can suppose $\pi_2$ contains the \orsl\  $m=\{M_b=(0,1,b)\st b\in\GF(q)\cup\{\infty\}\}$ of $\pi_1$.  Suppose that $\pi_2$ and $\pi_1$ contain a further common point $P$ not on $m$, so $P=(\mu,\nu,\omega)$ with $\mu,\nu,\omega\in\GF(q^3)$, with $\mu\ne 0$. As $P\in\pi_1$, we have
$\nu/\mu,\omega/\mu\in\GF(q)$.

Let $Q=(0,0,1)=M_\infty\in m$, so the line $PQ$ contains an \orsl\ of $\pi_2$. Without loss of generality, by suitable choice of $\mu,\nu,\omega\in\GF(q^3)$,  let that \orsl\ be determined by the three points $P,\ Q$ and $T=P+Q=(\mu,\nu,\omega+1)$.  That is,  $\pi_2$ is  determined by the  quadrangle $(0,1,1)$, $(0,1,0)$, $P=(\mu,\nu,\omega)$ and $T=(\mu,\nu,\omega+1)$. Further, $\pi_1$ and $\pi_2$ both contain the \orsl\ $m$ and the point $P$. 

We will show that, if $\ES_2=\ES_1$, then  $\mu\in\GF(q)$, and hence $\nu,\omega\in\GF(q)$ and so $T\in\pi_1$. 
The \orsl\ $n$ defined by $P,Q,T=P+Q$ has points  $N_c
=P+cQ=(\mu,\nu,\omega+c)$ for $c\in\GF(q)$. As the points $M_b,N_c$ lie in $\pi_2$ for $b,c\in\GF(q)$, the line $M_bN_c$ 
meets $\ell$ in a point $X_{bc}$ of $\ES_2$. Now $M_bN_c$ 
has coordinates $[c-b\nu+\omega,\ b \mu ,\ -\mu]$, and it meets $\ell$ in the point $X_{bc}=\car_1+\theta_{bc}\car_2$ with
\[
\theta_{bc}
=-\frac{c+b(\mu\tau-\nu)-(\mu\tau^2-\omega)}{c+b(\mu\tau^q-\nu)-(\mu\tau^{2q}-\omega)}.
\]
The point $X_{bc}$ is in the exterior splash $\ES_1$ if and only if $\theta_{bc}^{q^2+q+1}=-1$. We can regard this equation as a multivariate polynomial of degree 3 in $b$ and $c$. Note that if a multivariate polynomial of degree $k<q$ over $\GF(q^3)$ is identically zero over $\GF(q)$, then the polynomial is the zero polynomial. If $q=2$ or $3$, a computer search using Magma~\cite{magma} verifies the result holds, so we assume $q>3$. 
Expanding, and calculating the coefficients of $bc^2$ and $c^2$, and equating  them to zero  gives
\begin{eqnarray*}
0&=&\kappa\tau+\kappa^q\tau^q+\kappa^{q^2}\tau^{q^2}\\
0&=&\kappa\tau^2+\kappa^q\tau^{2q}+\kappa^{q^2}\tau^{2q^2},
\end{eqnarray*}
respectively, where $\kappa=\mu-\mu^{q^2}$, and so $\kappa+\kappa^q+\kappa^{q^2}=0$. That is,  the point $M_bN_c\cap\ell$ lies in the exterior splash  $\ES_1$ for all $b,c\in\GF(q)$ if 
\[\begin{pmatrix}
1&1&1\\
\tau&\tau^q&\tau^{q^2}\\
\tau^2&\tau^{2q}&\tau^{2q^2}\end{pmatrix}
\begin{pmatrix}\kappa\\ \kappa^q\\ \kappa^{q^2}\end{pmatrix}
=\begin{pmatrix}0\\0\\0\end{pmatrix}.
\]
This $3\times 3$ matrix  is invertible,  so the only solution is $\kappa=\kappa^q=\kappa^{q^2}=0$, that is $\mu=\mu^q$, and so $\mu\in\GF(q)$. Hence $\pi_2$ has exterior splash $\ES_1$ if and only if $\mu\in\GF(q)$.
As $\nu/\mu,\omega/\mu\in\GF(q)$, it  follows that $T\in\pi_2\cap\pi_1$. Thus $\pi_2$ and $\pi_1$ share the quadrangle $P,T,(0,1,0),(0,1,1)$ and so $\pi_2=\pi_1$ as required.
\end{proof}

Suppose $\pi_1,\pi_2$ are two exterior \orsps\ with a common exterior splash $\ES$ that meet in $q+1$ points. 
We conject that the two families of \orsls\ of $\ES$ with respect to $\pi_1,\pi_2$ are 
\emph{swapped}. That is, $\pi_1,\pi_2$ correspond to the  \orsps\ $\pi$, $\pi'$ with \orsls\ behaving as in  Theorem~\ref{orsls-same}.

\section{Conclusion}\Label{sec:con}

In this article we looked at the exterior splash as a set of points of $\PG(1,q^3)$, and showed its equivalence to $\GF(q)$-linear sets of rank 3 and size $q^2+q+1$, Sherk surfaces of size $q^2+q+1$, and covers of the circle geometry $CG(3,q)$.
We also investigated properties of an exterior \orsp\ and its corresponding exterior splash, in particular relating \orsls\ of an exterior splash to the corresponding \orsp. Further, we looked at how our construction of exterior splashes related to the projection construction of linear sets.

In 
future work,
we look in the Bruck-Bose representation of $\PG(2,q^3)$ in $\PG(6,q)$, and study exterior splashes in this context, further utilising properties of scattered linear sets of rank 3, as well as properties of hyper-reguli in $\PG(5,q)$.

\end{document}